\documentclass{amsart}
\usepackage{mathptmx}

\usepackage{amsmath}
\usepackage{amscd}
\usepackage{amssymb}
\usepackage{enumerate}
\usepackage{hyperref}

\DeclareMathOperator{\End}{End}

\DeclareMathOperator{\Ann}{Ann}

\DeclareMathOperator{\gr}{gr}

\DeclareMathOperator{\coker}{coker}

\newcommand{\D}{\mathcal{D}}

\newtheorem{thm}{Theorem}[section]

\newtheorem{prop}[thm]{Proposition}
\newtheorem{lem}[thm]{Lemma}
\newtheorem{cor}[thm]{Corollary}

\theoremstyle{definition}
\newtheorem{definition}[thm]{Definition}

\theoremstyle{remark}
\newtheorem{remark}[thm]{Remark}

\theoremstyle{convention}

\begin{document}

\title[Holonomic $\D$-modules over power series rings]{Van den Essen's theorem on the de Rham cohomology of a holonomic $\D$-module over a formal power series ring}
\author{Nicholas Switala}
\address{School of Mathematics\\ University of Minnesota\\ 127 Vincent Hall\\ 206 Church St. SE\\
Minneapolis, MN 55455}
\email{swit0020@math.umn.edu}
\thanks{NSF support through grant DMS-0701127 is gratefully acknowledged.}

\begin{abstract}
In this expository paper, we give a complete proof of van den Essen's theorem that the de Rham cohomology spaces of a holonomic $\D$-module are finite-dimensional in the case of a formal power series ring over a field of characteristic zero.  This proof requires results from at least five of van den Essen's papers as well as his unpublished thesis, and until now has not been available in a self-contained document.
\end{abstract}

\maketitle

\section{Introduction}\label{intro} 

Let $k$ be a field of characteristic zero, let $R = k[x_1, \ldots, x_n]$ be a polynomial ring over $k$, and let $\D = \D(R,k)$ be the ring of $k$-linear differential operators on $R$ (the \emph{Weyl algebra}).  To any finitely generated left $\D$-module $M$, we can associate its \emph{dimension} $d(M)$: if $M$ is nonzero, this dimension is an integer between $n$ and $2n$.  The left $\D$-modules of minimal dimension (those for which $d(M) = n$) are called \emph{holonomic}.  A basic result in the theory of $\D$-modules, due to Bernstein, states that the de Rham cohomology spaces of a holonomic left $\D$-module $M$ are finite-dimensional over $k$.  These spaces are the cohomology objects of a complex defined using the usual exterior derivative formulas with respect to the action of the partial derivatives $\partial_1, \ldots, \partial_n \in \D$ on $M$.  The key idea in the proof of this finiteness is that the kernel and cokernel of $\partial_n$ acting on $M$ are holonomic $\D_{n-1}$-modules, where $\D_{n-1} = \D(k[x_1, \ldots, x_{n-1}],k)$; with this statement in hand, the finiteness of the de Rham cohomology follows by a routine induction.   

Now consider the case where $R$ is a formal power series ring $k[[x_1, \ldots, x_n]]$, again over a field of characteristic zero.  We again have the ring $\D = \D(R,k)$ of $k$-linear differential operators on $R$, and a well-defined notion of dimension for finitely generated left $\D$-modules (hence a notion of holonomy for left $\D$-modules).  In this case, the analogue of Bernstein's result is due to van den Essen.  If $M$ is a holonomic left $\D$-module, its de Rham cohomology spaces are again finite-dimensional over $k$, just as in the polynomial case; in contrast to this case, however, it is not true in general that the cokernel of $\partial_n$ acting on $M$ is a holonomic $\D_{n-1}$-module, which makes the proof more difficult.  The kernel of $\partial_n$ is again holonomic, and the cokernel is holonomic whenever $M$ satisfies a certain generic condition called \emph{$x_n$-regularity}.  It turns out that if $M$ is holonomic, we can always make a linear change of coordinates (which does not affect de Rham cohomology) after which $M$ becomes $x_n$-regular.  The same routine induction argument used by Bernstein is then sufficient to prove finiteness of the de Rham cohomology in the formal power series case as well:

\begin{thm}\cite[Prop. 2.2]{essen}\label{mainthm}
Let $k$ be a field of characteristic zero, let $R = k[[x_1, \ldots, x_n]]$ be a formal power series ring over $k$, and let $\D = \D(R,k)$ be the ring of $k$-linear differential operators on $R$.  If $M$ is a holonomic left $\D$-module, its de Rham cohomology spaces $H^i_{dR}(M)$ are finite-dimensional over $k$ for all $i$.
\end{thm}

Van den Essen's proof is not contained in a single paper.  The complete argument requires results from at least five of his papers, as well as his (unpublished) thesis.  Moreover, some of the necessary results are proved more than once in these papers, with simpler and better proofs superseding more complicated ones.  The purpose of this expository paper is to assemble these preliminary results and proofs in one place, giving only the shortest argument in each case.  

Ideally, this paper would be entirely self-contained except for basic commutative algebra, but the amount of necessary background material on $\D$-modules is too large for this ideal to be reasonable.  Our compromise is the following.  Bj\"{o}rk's book \cite{bjork} is our basic reference for the theory of $\D$-modules (it is also the basic reference cited in van den Essen's papers), and we freely quote without proof results appearing in this book.  We will also appeal to Gabber's deep result \cite[Thm. I]{gabber} on the involutivity of characteristic ideals without providing a proof.  We will, however, give full proofs for all preliminary results taken from van den Essen's papers.  We stress that nothing in this paper is original, neither the results nor the proofs; our goal in writing it is merely to make available a complete proof of Theorem \ref{mainthm} in one document.       

In this paper, we only state and prove precisely what we need for Theorem \ref{mainthm}.  The papers of van den Essen cited here contain many more results on kernels and cokernels of differential operators that are not strictly necessary for the proof of this theorem, and we encourage the interested reader to investigate further.

The structure of this paper is as follows.  In section \ref{prelim}, we collect preliminary material on formal power series rings, $\D$-modules, de Rham cohomology, and Gabber's theorem.  In section \ref{kernels}, we give the proof that the kernel of $\partial_n$ acting on a holonomic $\D$-module is again holonomic (with no further conditions on the module).  In section \ref{regularity}, we define the $x_n$-regularity condition for a holonomic $\D$-module and prove some technical results concerning the consequences of this condition.  In section \ref{cokernels}, we give the proof that (possibly after a linear change of coordinates) the cokernel of $\partial_n$ acting on a holonomic $\D$-module is again holonomic, and then complete the proof of Theorem \ref{mainthm}.  

\section{Preliminaries}\label{prelim}

Throughout the paper, $k$ denotes a field of characteristic zero, $R$ denotes the formal power series ring $k[[x_1, \ldots, x_n]]$, and $R_{n-1}$ denotes the subring $k[[x_1, \ldots, x_{n-1}]]$.  The rings $R$ and $R_{n-1}$ are commutative, Noetherian, regular local rings.  We denote by $\mathfrak{m}$ the unique maximal ideal $(x_1, \ldots, x_n)$ of $R$ (similarly, $\mathfrak{m}_{n-1}$ is the unique maximal ideal of $R_{n-1}$).  Since $R$ is local, any element of $R$ with a nonzero constant term is a unit.

\begin{definition}\label{xnregf}
A formal power series $f \in R$ is said to be \emph{$x_n$-regular} if $f(0,0,\ldots, 0, x_n) \neq 0$ in $k[[x_n]]$, that is, if a term $c_{0,\ldots, 0, i}x_n^i$ with $c_{0,\ldots, 0, i} \in k \setminus \{0\}$ occurs in $f$.
\end{definition}

The following theorem clarifies the significance of the $x_n$-regularity hypothesis:

\begin{thm}[Weierstrass preparation theorem]\cite[Thm. IV.9.2]{lang}\label{weierprep}
Suppose that $f \in R$ is $x_n$-regular.  There exists a unique expression $f = u(x_n^d + b_{n-1}x_n^{d-1} + \cdots + b_0)$ where $u \in R$ is a unit and each $b_i \in \mathfrak{m}_{n-1}$.
\end{thm}

\begin{remark}\label{wpfingen}
The Weierstrass preparation theorem has the following consequence: if $f \in R$ is $x_n$-regular, then $R/fR$ is finitely generated (by the classes of $x_n^i$ with $0 \leq i \leq d-1$) as a module over $R_{n-1}$.  It follows that \emph{any} finitely generated $R/fR$-module is in fact a finitely generated $R_{n-1}$-module.  In the sequel, our appeals to the ``Weierstrass preparation theorem'' will in fact be appeals to this consequence.
\end{remark}

We now review some definitions and properties of $\D$-modules and de Rham complexes.  Our basic reference for what follows is \cite{bjork}.  The ring $\D = \D(R,k)$ of $k$-linear differential operators, a subring of $\End_k(R)$, takes the form $\D = R\langle \partial_1, \ldots, \partial_n \rangle$, where $\partial_i = \frac{\partial}{\partial x_i}$.  (This notation is meant to indicate that, after adjoining the new variables $\partial_i$ to $R$, we do \emph{not} obtain a commutative ring.) As a left $R$-module, $\D$ is free on monomials in the $\partial_i$.  If $R$ is any commutative ring and $A \subset R$ any commutative subring, there is a more general definition (\cite[\S 16]{EGAIV}) of the ring $\D(R,A)$ of $A$-linear differential operators on $R$.  See \cite[Thm. 16.11.2]{EGAIV} for a proof that our definition coincides with this more general one in the formal power series case.

Unless expressly indicated otherwise, by a \emph{$\D$-module} we will always mean a \emph{left} module over $\D$.

The ring $\D$ has an increasing, exhaustive filtration $\{\D_j\}$, called the \emph{order filtration}, where $\D_j$ is the $R$-submodule consisting of those differential operators of order $\leq j$ (the order of an element of $\D$ is the maximum of the orders of its summands, and the order of a single summand $\rho \partial_1^{a_1} \cdots \partial_n^{a_n}$ with $\rho \in R$ is $\sum a_i$).  Note that for all $f \in R$ and for all $i$, we have the relation 
\[
[\partial_i, f] = \partial_i f - f \partial_i = \partial_i(f) \in R \subset \D,
\]
where $[,]$ denotes the commutator of two elements of $\D$ and the operator $\partial_i(f) \in \D$ is \emph{multiplication} by $\partial_i(f) \in R$.  Consequently, the associated graded object $\gr \D = \oplus_j \D_j/\D_{j-1}$ with respect to the order filtration is isomorphic to the polynomial ring $R[\zeta_1, \ldots, \zeta_n]$, where $\zeta_i$ is the image of $\partial_i$ in $\D_1/\D_0 \subset \gr \D$.  (In particular, $\gr \D$ is commutative.) For all $i$, $\zeta_i$ is called the \emph{principal symbol} of $\partial_i$, and we write $\zeta_i = \sigma(\partial_i)$.  More generally, if $\delta \in \D_j \setminus \D_{j-1}$, its principal symbol $\sigma(\delta)$ is its class in $\gr^j \D = \D_j/\D_{j-1} \subset \gr \D$.  

If $M$ is a finitely generated $\D$-module, there exists an increasing, exhaustive filtration $\{M_j\}$ of $M$ such that $M$ becomes a filtered $\D$-module with respect to the order filtration on $\D$ (so $\D_i \cdot M_j \subset M_{i+j}$ for all $i$ and $j$) \emph{and} $\gr M = \oplus_j M_j/M_{j-1}$ is a finitely generated $\gr \D$-module.  We call such a filtration \emph{good}.  Let $J$ be the radical of $\Ann_{\gr \D} \gr M$ (the \emph{characteristic ideal} of $M$) and set $d(M) = \dim \, (\gr\D)/J$ where $\dim$ denotes Krull dimension.  The ideal $J$, and hence the number $d(M)$, is independent of the choice of good filtration on $M$.  By \emph{Bernstein's theorem}, if $M$ is a (nonzero) finitely generated $\D$-module, we have $n \leq d(M) \leq 2n$.  In the case where $d(M) = n$ we say that $M$ is \emph{holonomic}.  

Some basic facts about holonomic modules are the following: submodules and quotients of holonomic $\D$-modules are holonomic, an extension of a holonomic $\D$-module by another holonomic $\D$-module is holonomic, holonomic $\D$-modules are of finite length over $\D$, and holonomic $\D$-modules are cyclic (generated over $\D$ by a single element).  We will use these basic facts below without comment.

If $M$ is a $\D$-module, the operator $\partial_n \in \D$ acts on $M$ via a $\D_{n-1}$-linear map, and so its kernel and cokernel are $\D_{n-1}$-modules.  The main question we will be concerned with in this paper is the following: if $M$ is holonomic, are the kernel and cokernel of $\partial_n$ acting on $M$ holonomic $\D_{n-1}$-modules?

Exactly the same question can be asked about the operator $x_n \in \D$.  This question is easier, and we have the following unconditional affirmative answer:

\begin{prop}\cite[Thm. 3.4.2, Prop. 3.4.4]{bjork}\label{bjorkxn}
Let $M$ be a holonomic $\D$-module. The kernel and cokernel of $x_n$ acting on $M$ are holonomic $\D_{n-1}$-modules.
\end{prop}

We remark that in the polynomial ring case, $x_n$ and $\partial_n$ play essentially symmetric roles, and so the question for $\partial_n$ is no more difficult than the question for $x_n$ (and has the same unconditional affirmative answer).  The symmetry between $x_n$ and $\partial_n$ does not persist in the formal power series case, which is why the question for $\partial_n$ is significantly more difficult.

We next discuss the \emph{de Rham complex} of a $\D$-module $M$.  This is a complex of length $n$, denoted $M \otimes \Omega_R^{\bullet}$, whose objects are $R$-modules but whose differentials are merely $k$-linear.  It is defined as follows: for $0 \leq i \leq n$, $M \otimes \Omega^i_R$ is a direct sum of $n \choose i$ copies of $M$, indexed by $i$-tuples $1 \leq j_1 < \cdots < j_i \leq n$.  The summand corresponding to such an $i$-tuple will be written $M \, dx_{j_1} \wedge \cdots \wedge dx_{j_i}$.  

The $k$-linear differentials $d^i: M \otimes \Omega_R^i \rightarrow M \otimes \Omega_R^{i+1}$ are defined by 
\[
	d^i(m \,dx_{j_1} \wedge \cdots \wedge dx_{j_i}) = \sum_{s=1}^n \partial_s(m)\, dx_s \wedge dx_{j_1} \wedge \cdots \wedge dx_{j_i},
\]
with the usual exterior algebra conventions for rearranging the wedge terms, and extended by linearity to the direct sum.  The cohomology objects $h^i(M \otimes \Omega_R^{\bullet})$, which are $k$-spaces, are called the \emph{de Rham cohomology spaces} of the $\D$-module $M$, and are denoted $H^i_{dR}(M)$.  

We have defined the de Rham complex of a $\D$-module using a chosen regular system of parameters $x_1, \ldots, x_n$ for the formal power series ring $R$.  There is an alternate definition from which it is easier to see that this complex does not depend on the chosen parameters, based on the characterization of $\D$-modules in terms of \emph{integrable connections}.  Let $\Omega_R^1$ be the $R$-module of ($\mathfrak{m}$-adically) \emph{continuous} K\"{a}hler differentials of $R$ over $k$ \cite[20.7.14]{EGAIVa}, and $d: R \rightarrow \Omega_R^1$ the corresponding universal continuous derivation.  In coordinates, if $x_1, \ldots, x_n$ is a regular system of parameters for $R$, we have $\Omega_R^1 \simeq \oplus_i R \, dx_i$ and $d(f) = \sum_i \partial_i(f) \, dx_i$ for all $f \in R$.  However, $\Omega_R^1$ and $d$ can also be defined using a universal property, with no reference to coordinates: every $\mathfrak{m}$-adically continuous derivation $\delta: R \rightarrow M$ where $M$ is an $R$-module factors uniquely through $d$.  Now recall that a \emph{connection} on an $R$-module $M$ is a $k$-linear map $\nabla: M \rightarrow \Omega_R^1 \otimes_R M$ such that $\nabla(rm) = dr \otimes m + r \cdot \nabla(m)$ for all $r \in R$ and $m \in M$.  A connection $\nabla = \nabla^0$ on $M$ induces $k$-linear maps $\nabla^l: \Omega_R^l \otimes_R M \rightarrow \Omega_R^{l+1} \otimes_R M$ for all $l \geq 0$, where $\Omega_R^l$ is the $l$th exterior power of $\Omega_R^1$.  If $\nabla^1 \circ \nabla^0$ is the zero map, the connection $\nabla$ is said to be \emph{integrable}, in which case $\nabla^{l+1} \circ \nabla^l$ is the zero map for all $l$.  Since $\D$ is generated over $R$ by derivations, the data of a left $\D$-module structure on an $R$-module $M$ is equivalent to that of an integrable connection on $M$ \cite[3.2.9]{bjork}, and the complex $(\Omega_R^{\bullet} \otimes_R M, \nabla^{\bullet})$ induced by $\nabla$ is the \emph{de Rham complex} of $M$, which in coordinates $\{x_i\}$ is exactly the complex $M \otimes \Omega^{\bullet}_R$ defined above.  The only use we will have in this paper for this alternate definition of $M \otimes \Omega^{\bullet}_R$ is the following obvious consequence:

\begin{prop}\label{dRind}
The de Rham cohomology spaces $H^i_{dR}(M)$ of any $\D$-module $M$ are independent of the choice of a regular system of parameters $x_1, \ldots, x_n$ for $R$.
\end{prop}

There is a long exact sequence relating the de Rham cohomology of a $\D$-module $M$ with the de Rham cohomology of the kernel and cokernel of $\partial_n$ acting on $M$ (which are $\D_{n-1}$-modules):

\begin{lem}\cite[Prop. 2.4.13]{bjork}\label{derhamles}
Let $M$ be a $\D$-module.  Let $M_*$ (resp. $\overline{M}$) be the kernel (resp. cokernel) of $\partial_n$ acting on $M$.  Then there is a long exact sequence
\[
\cdots \rightarrow H^{i-2}_{dR}(\overline{M}) \rightarrow H^i_{dR}(M_*) \rightarrow H^i_{dR}(M) \rightarrow H^{i-1}_{dR}(\overline{M}) \rightarrow \cdots
\]
of $k$-spaces, where $H^j_{dR}(M_*)$ and $H^j_{dR}(\overline{M})$ are de Rham cohomology spaces of the $\D_{n-1}$-modules $M_*$ and $\overline{M}$, defined using only $\partial_1, \ldots, \partial_{n-1}$.
\end{lem}

Finally, we will need Gabber's theorem on involutivity of characteristic ideals, originally proved in \cite{gabber}.  We need to introduce the Poisson bracket (see \cite{gabber} or \cite[App. D]{hotta}).  Its definition makes sense, and Gabber's theorem holds, for more general filtered rings, but we content ourselves here with stating everything for the ring $\D$ and its order filtration $\{\D_j\}$.  Suppose that $\delta \in \gr^i \D$ and $\delta' \in \gr^j \D$ for some $i$ and $j$.  Then we can write $\delta = \sigma(d)$ and $\delta' = \sigma(d')$ for some $d \in \D_i$ and $d' \in \D_j$, where $\sigma$ denotes the principal symbol.  Since $\gr \D$ is commutative, the commutator $[d,d']$ belongs to $\D_{i+j-1}$, and we define $\{\delta, \delta'\} \in \gr^{i+j-1} \D$ to be the principal symbol of $[d,d']$.  It is easy to check that $\{\delta, \delta'\}$ does not depend on the choices of $d$ and $d'$.  The unique biadditive extension
\[
\{,\}: \gr \D \times \gr \D \rightarrow \gr \D   
\]
is called the \emph{Poisson bracket} on $\gr \D$, and, in particular, is a biderivation.  An ideal $I \subset \gr \D$ is called \emph{involutive} if it is closed under the Poisson bracket, that is, if $\{I,I\} \subset I$.

\begin{thm}[Gabber's theorem]\cite[Thm. I]{gabber}\label{gabber}
Let $M$ be a finitely generated left $\D$-module.  Let $\gr M$ be the associated graded $\gr \D$-module of $M$ with respect to a chosen good filtration on $M$.  Let $J = \sqrt{\Ann_{\gr \D} \gr M} \subset \gr \D$ be the \emph{characteristic ideal} of $M$ (as stated earlier, $J$ does not depend on the choice of good filtration).  Then $J$ is involutive.
\end{thm}

\begin{cor}\cite[Lemma 1.12]{essen}\label{primeinvol}
With the notation of Theorem \ref{gabber}, let $\mathfrak{p}$ be a minimal prime ideal over $J$.  Then $\mathfrak{p}$ is again involutive.
\end{cor}

\begin{proof}
As $J$ is a radical ideal in a Noetherian ring, we can write $J$ as an intersection $J = \mathfrak{p}_1 \cap \cdots \cap \mathfrak{p}_t$ of prime ideals of $\gr \D$, and we may assume that each $\mathfrak{p}_i$ is minimal over $J$.  If $t=1$, $J$ is prime and there is nothing to prove.  Otherwise, fix $i$, write $\mathfrak{p}$ for $\mathfrak{p}_i$, and let $\mathfrak{q} = \cap_{j \neq i} \mathfrak{p}_j$.  There exists some $c \in \mathfrak{q} \setminus \mathfrak{p}$.  Let $a, b \in \mathfrak{p}$ be given.  We have $ac, bc \in \cap_j \mathfrak{p}_j = J$, and hence $\{ac,bc\} \in J$, since $J$ is involutive by Theorem \ref{gabber}.  We now use the fact that the Poisson bracket is a biderivation, which gives
\[
\{ac,bc\} = a\{c,b\}c + a\{c,c\}b + c\{a,b\}c + c\{a,c\}b.
\]
The first, second, and fourth summands on the right-hand side all have a factor of $a$ or $b$ and thus belong to $\mathfrak{p}$.  Since the sum belongs to $\mathfrak{p}$, the third summand, $c^2\{a,b\}$, belongs to $\mathfrak{p}$ as well.  As $\mathfrak{p}$ is prime and $c^2 \notin \mathfrak{p}$, we must have $\{a,b\} \in \mathfrak{p}$, and so $\mathfrak{p}$ is involutive.
\end{proof}

\begin{remark}\label{zetahom}
If $M$, $J$, and $\mathfrak{p}$ are as above, we note that $J$ is homogeneous with respect to $\zeta_1, \ldots, \zeta_n$ by definition, and by \cite[Thm. 13.7(i)]{matsumura}, $\mathfrak{p}$ is homogeneous with respect to the $\zeta_i$ as well.
\end{remark}

\section{Kernels}\label{kernels}

In this section, we prove that the kernel $M_*$ of $\partial_n$ acting on a holonomic $\D$-module $M$ is a holonomic $\D_{n-1}$-module, with no further conditions on $M$.  After some reductions, it suffices to check that the \emph{cokernel} of $x_n$ is holonomic, which is Proposition \ref{bjorkxn}.  These reductions are made possible by the following key lemma, which states that $R$-linear dependence relations among elements of $M_*$ hold homogeneously in $x_n$.  (We write $\partial$ for $\partial_n$; this shorthand will be used throughout the rest of the paper.)

\begin{lem}[Key lemma for kernels]\cite[Lemme 1]{kernel}\label{kerlemma}
Let $M$ be any $\D$-module, and let $M_* = \ker(\partial: M \rightarrow M)$.  Suppose that $m_1, \ldots, m_l \in M_*$ are such that $f_1m_1 + \cdots + f_lm_l = 0$ for some $f_1, \ldots, f_l \in R$.  Then $f_{1,j}m_1 + \cdots + f_{l,j}m_l = 0$ for every $j \geq 0$, where $f_{i,j} \in R_{n-1}$ denotes the coefficient of $x_n^j$ in $f_i$.
\end{lem}

\begin{proof}
We first assume that the statement holds for $j = 0$ and prove that it follows for all $j > 0$.  Note that for any $j \geq 0$ and any $f \in R$, we have $f_j = \frac{1}{j!}(\partial^j f)_0$, where the subscript $0$ denotes the constant term with respect to $x_n$.  If $m_1, \ldots, m_l \in M_*$ and $f_1m_1 + \cdots + f_lm_l = 0$, then $\partial^j(f_1m_1 + \cdots + f_lm_l) = 0$.  By the Leibniz rule, 
\[
\partial^j(\sum_{i=1}^l f_im_i) = \sum_{i=1}^l (\partial^j(f_i)) m_i,
\]
as all other terms have a factor of $\partial^{\alpha}(m_i)$ for some $\alpha > 0$ and some $i$ and hence vanish since $m_i \in M_*$.  Multiplying by a harmless constant, we see that $\frac{1}{j!} \sum_{i=1}^l (\partial^j(f_i)) m_i = 0$.  By our assumption, $\sum_{i=1}^l (\frac{1}{j!}\partial^j(f_i))_0 m_i = 0$, but this sum is nothing but $\sum_{i=1}^l f_{i,j}m_i$.  We have thus reduced ourselves to the case $j=0$.   

Let $m = \sum_{i=1}^l f_{i,0} m_i$, and let $E$ be the $R$-submodule of $M$ generated by $\{m_1, \ldots, m_l\}$.  We claim that $m \in x_n^q E$ for all $q \geq 1$.  As $E$ is a finitely generated $R$-module and $R$ is a Noetherian local domain (whose maximal ideal contains $x_n$), this will imply that $m = 0$ \cite[Thm. 8.10(ii)]{matsumura}, as desired.  We prove $m \in x_n^q E$ for all $q$ by induction on $q$.  For the base case, $q=1$, consider the difference $\sum_{i=1}^l f_i m_i - \sum_{i=1}^l f_{i,0}m_i$.  On the one hand, by definition, this difference belongs to $x_nE$, as we have removed all terms which \textit{a priori} may lack an $x_n$ factor.  On the other hand, the first term is $0$ by hypothesis and the second is $m$, so $-m \in x_nE$.  Now suppose that $m \in x_n^qE$ for some $q \geq 1$, and write $m = x_n^q \sum_{i=1}^l r_im_i$ for some $r_1, \ldots, r_l \in R$.  As $m$ is an $R_{n-1}$-linear combination of elements of $M$ killed by $\partial$, we have $\partial(m) = 0$, so $\frac{1}{q!}x_n^q\partial^q m = 0$.  Substituting $x_n^q \sum_{i=1}^l r_im_i$ for $m$ in the left-hand side of this equation and using the Leibniz rule, we see that the only terms which \textit{a priori} may lack an $x_n^{q+1}$ factor are those in the sum $\sum_{i=1}^l r_{i,0}x_n^qm_i$: we obtain, for some $\mu \in E$, an expression
\[
0 = \frac{1}{q!}x_n^q\partial^q m = \frac{1}{q!}x_n^q\partial^q(x_n^q \sum_{i=1}^l r_im_i) = \sum_{i=1}^l r_{i,0}x_n^qm_i + x_n^{q+1}\mu,
\]
so that $\sum_{i=1}^l r_{i,0}x_n^qm_i \in x_n^{q+1}E$.  It follows that $m \in x_n^{q+1}E$, as 
\[
m - \sum_{i=1}^l r_{i,0}x_n^qm_i = \sum_{i=1}^l(\sum_{s=1}^{\infty} r_{i,s}x_n^{s+q}m_i)
\] 
and we have an $x_n^{q+1}$ factor in every remaining term.  This completes the proof.
\end{proof}

\begin{prop}\cite[Thm. (iii)]{kernel}\label{holkernel}
Let $M$ be a holonomic $\D$-module.  Then $M_*$ is a holonomic $\D_{n-1}$-module.
\end{prop}

\begin{proof}
Consider the $R$-submodule $R \cdot M_*$ of $M$ generated by $M_*$.  If $r \in R$ and $m \in M_*$, then $\partial(rm) = \partial(r)m + r\partial(m) = \partial(r)m$; since $M_*$ is already a $\D_{n-1}$-submodule of $M$, this calculation shows that $R \cdot M_*$ is a $\D$-submodule of $M$.  Furthermore, it is clear that $\ker(\partial: R \cdot M_* \rightarrow R \cdot M_*)$ coincides with $M_*$.  Therefore, we may assume that $M = R \cdot M_*$.  With this assumption, we conclude at once that $M = M_* + x_nM$ as $\D_{n-1}$-modules, since $M_*$ is an $R_{n-1}$-module.  We claim that this sum is direct.  Suppose that $m \in M_* \cap x_nM$.  Since $x_nM = x_n(R \cdot M_*)$, we can write $m = x_n(\sum r_im_i)$ for some $r_i \in R$ and $m_i \in M_*$.  From this, we obtain an equation $m - x_n(\sum r_im_i) = 0$ where $m$ and all the $m_i$ belong to $M_*$.  By Lemma \ref{kerlemma}, the constant term of the left-hand side with respect to $x_n$, which is simply $m$, also vanishes.  Thus $M_* \cap x_nM = 0$, and so $M = M_* \oplus x_nM$.  This direct sum decomposition implies that $M_* \simeq M/x_nM$, which is a holonomic $\D_{n-1}$-module by Proposition \ref{bjorkxn}.  This completes the proof.
\end{proof}

\section{Regularity}\label{regularity}

In this section, we define what it means for a $\D$-module $M$ to be \emph{$x_n$-regular}.  If $M$ is holonomic, this is a weak condition that is always satisfied up to a linear change of variables in $R$.  This is the technical assumption necessary for the cokernel of $\partial$ acting on a holonomic $\D$-module to again be holonomic.

\begin{definition}\cite[p. 21]{essthesis}\label{E}
Let $M$ be a $\D$-module, let $m \in M$, and suppose that $\tau \in \D$ is a $k$-linear derivation.  We write $E_{\tau}(m)$ for the $R$-submodule $\sum_{i=0}^{\infty} R \cdot \tau^i(m)$ of $M$ generated by the family $\{\tau^i(m)\}$.  If $N \subset M$ is an $R$-submodule, $E_{\tau}(N)$ is the $R$-submodule of $M$ generated by $E_{\tau}(n)$ for $n \in N$.
\end{definition}

For a given $\tau$ and $m$, if $E_{\tau}(m)$ is a finitely generated $R$-module, then there exists $p$ such that $\tau^p(m) \in \sum_{i=0}^{p-1} R \cdot \tau^i(m)$.  In this case, $E_{\tau}(m) = \sum_{i=0}^{p-1} R \cdot \tau^i(m)$.

\begin{lem}\cite[Ch. II, Prop. 1.16]{essthesis}\label{reglink}
Let $M$ be a $\D$-module, and suppose there exists a nonzero $f \in R$ such that the localization $M_f$ is a finitely generated $R_f$-module.  Then for any $m \in M$, there exists $s \geq 0$ such that $E_{f^s \partial}(m)$ is a finitely generated $R$-module.
\end{lem}

\begin{proof}
Let $m \in M$ be given.  Since $R$ (and hence $R_f$) is Noetherian, any $R$-submodule of $M_f$ is also finitely generated over $R$.  In particular, this is true of the $R$-submodule $E_{\partial}(R_f \cdot m) \subset M_f$, where here we are regarding $M_f$ as a $\D$-module in the obvious way (and replacing $R$ with $R_f$ in our definition of the $E_{\partial}$ construction).  Therefore, for some $p \geq 1$, $E_{\partial}(R_f \cdot m) = R_f \cdot m + R_f \cdot \partial(m) + \cdots + R_f \cdot \partial^{p-1}(m)$, and consequently $\partial^p(m)$ can be written $\rho_0 m + \rho_1 \partial(m) + \cdots + \rho_{p-1} \partial^{p-1}(m)$ for some $\rho_0, \ldots, \rho_{p-1} \in R_f$.  Clearing denominators (and multiplying by a further power of $f$ if necessary), we see that there exists $s \geq 0$ such that $f^s \partial^p(m) = r_0 m + r_1 \partial(m) + \cdots r_{p-1} \partial^{p-1}(m)$ for some $r_0, \ldots, r_{p-1} \in R$.  Write $N$ for the $R$-submodule $\sum_{i=0}^{p-1} R \cdot \partial^i(m)$ of $M$.  Then the fact that $f^s \partial^p(m) \in N$ implies that $f^s \partial(N) \subset N$, from which it follows at once that $E_{f^s \partial}(N) \subset N$.  By definition, $N$ is a finitely generated $R$-module, so $E_{f^s \partial}(N)$ and its $R$-submodule $E_{f^s \partial}(m)$ are finitely generated $R$-modules, completing the proof.
\end{proof}

Lemma \ref{reglink} is useful because its hypothesis is satisfied for \textit{every} holonomic $\D$-module.

\begin{prop}\cite[Prop. 1]{example}\label{fgloc}
Let $M$ be a holonomic $\D$-module.  There exists a nonzero $f \in R$ such that $M_f$ is a finitely generated $R_f$-module.
\end{prop}

\begin{proof}
We first consider two special cases.  If $M$ is $R$-torsionfree, this is \cite[Lemma 3.3.19]{bjork}.  If $M$ is simple as a $\D$-module and is not $R$-torsionfree, there exist nonzero $f \in R$ and $m \in M$ such that $fm = 0$.  By the simplicity of $M$, $M = \D \cdot m$, from which it follows that $M_f$ is zero and \textit{a fortiori} finitely generated.  To prove the proposition for all holonomic $\D$-modules $M$, we use the fact that any such $M$ is of finite length as a $\D$-module.  If  
\[
0 \rightarrow M' \rightarrow M \rightarrow M'' \rightarrow 0
\]
is a short exact sequence of $\D$-modules and there exist nonzero $g,h \in R$ such that $M'_f$ (resp. $M''_g$) is a finitely generated $R_f$- (resp. $R_g$-) module, then $M_{fg}$ is a finitely generated $R_{fg}$-module.  Therefore induction on length, which reduces us to the case of a simple $\D$-module and hence to the previous two special cases, completes the proof.
\end{proof}

Recall from section \ref{prelim} that $f \in R$ is said to be \emph{$x_n$-regular} if $f(0,0,\ldots, 0, x_n) \neq 0$, in which case $f$ can be written as the product of a unit and a ``Weierstrass polynomial'' in $x_n$.

\begin{definition}\cite[p. 903]{cokernel}\label{Mreg}
Let $M$ be a $\D$-module.  We say that $M$ is \emph{$x_n$-regular} if for every $m \in M$, there exists an $x_n$-regular $f \in R$ such that $E_{f \partial}(m)$ is a finitely generated $R$-module.  (Any $m \in M$ for which this holds is said to be an \emph{$x_n$-regular element}.) 
\end{definition}

\begin{remark}\label{holMreg}
In what follows, we will frequently consider $\D$-modules $M$ which are holonomic, hence cyclic, together with a choice of $m$ such that $M = \D \cdot m$.  It is not hard to check that if $E_{\tau}(m)$ is a finitely generated $R$-module for some derivation $\tau \in \D$, then $E_{\tau}(\delta(m))$ is also finitely generated for every $\delta \in \D$.  It follows that if $M = \D \cdot m$ is cyclic, then $M$ is $x_n$-regular as long as $E_{\tau}(m)$ is finitely generated over $R$, and this does not depend on the choice of $\D$-module generator $m$ for $M$.
\end{remark}

Having now stated all necessary definitions, we begin working toward the crucial technical result (Lemma \ref{coklemma}) on the cokernel of $\partial$ acting on a holonomic, $x_n$-regular $\D$-module.  The following proposition can be viewed as a generalization of the Weierstrass preparation theorem (which we recover by taking $l=0$):

\begin{prop}\cite[Thm. 1]{several2}\label{weiergen}
Let $\Delta: R \rightarrow R$ be a differential operator of the form $\Delta = \sum_{i=0}^{l} r_i \partial^i$ where $r_i \in R$ for all $i$ and $r_l$ is $x_n$-regular.  Then $R/\Delta(R)$ is a finitely generated $R_{n-1}$-module.
\end{prop}

We state and prove a special case of Proposition \ref{weiergen}, due to Malgrange, separately:

\begin{lem}\cite[Prop. 1.3]{malgrange}\label{malglem}
Let $R = k[[x]]$ and let $\Delta: R \rightarrow R$ be a nonzero differential operator.  Then $R/\Delta(R)$ is a finite-dimensional $k$-space.
\end{lem}

\begin{proof}
For any formal power series $r = \sum \alpha_i x^i \in R$, we let $\nu(r) = \min{\{i | \alpha_i \neq 0\}}$.  The differential operator $\Delta$ takes the form $\Delta = \sum_{i=0}^{l} r_i \partial^i$ where $r_l \neq 0$.  Set $s = \max{\{i - \nu(r_i)\}}$, and let $I \subset \{0, \ldots, l\}$ be the set of indices for which this maximum is attained, that is, for which $s = i - \nu(r_i)$.  For each $i \in I$, $r_i = x^{i-s}\rho_i$ for some $\rho_i \in R$ such that $\rho_i(0) \neq 0$.  For any integer $t \geq s$, we have
\[
\Delta(x^t) = \sum_{i \in I} t(t-1)\cdots (t-i+1) \rho_i(0) x^{t-s} + \textrm{higher order terms}.
\]
The coefficient of $x^{t-s}$ in the above expression, namely $\sum_{i \in I} t(t-1)\cdots (t-i+1) \rho_i(0)$, is a polynomial in $t$ whose leading term is $\rho_{\max I}(0) t^{\max I}$.  Since $\max I \in I$, we have $\rho_{\max I}(0) \neq 0$, and so for large enough $t$, this leading term dominates the polynomial: there exists $t_0$ such that for all $t \geq t_0$,
\[
\Delta(x^t) = (\textrm{nonzero})x^{t-s} + \textrm{higher order terms}.
\]
Since $k$ is of characteristic zero, it follows that if $t \geq t_0$, given any $g \in R$ such that $\nu(g) \geq t-s$ (that is, $g \in \mathfrak{m}^{t-s}$, where $\mathfrak{m} \subset R$ is the maximal ideal), we can solve the equation $\Delta(f) = g$ uniquely for $f \in R$ by recursion on the coefficients, and the unique solution $f$ belongs to $\mathfrak{m}^t$.  Therefore the restriction of $\Delta$ is an isomorphism of $k$-spaces $\mathfrak{m}^t \xrightarrow{\sim} \mathfrak{m}^{t-s}$ for any $t \geq t_0$.  Fix such a $t$, and consider the commutative diagram
\[
\begin{CD}
0 @>>> \mathfrak{m}^t @>>> R @>>> R/\mathfrak{m}^t @>>> 0\\
@.            @VV \Delta V     @VV \Delta V   @VV \overline{\Delta} V    @.\\
0 @>>> \mathfrak{m}^{t-s} @>>> R @>>> R/\mathfrak{m}^{t-s} @>>> 0\\
\end{CD}
\]
of $k$-spaces with exact rows, where $\overline{\Delta}$ is the map induced on quotients by $\Delta$.  We know that the left vertical arrow is an isomorphism, and the source and target of the right vertical arrow are finite-dimensional $k$-spaces.  It follows at once from the snake lemma that the middle vertical arrow has finite-dimensional cokernel, as desired.
\end{proof}

\begin{proof}[Proof of Proposition \ref{weiergen}]
We proceed by induction on $n$.  In the base case $n = 1$, the hypothesis that $r_l$ be $x_n$-regular reduces to the hypothesis that $r_l \neq 0$, and so we are done by Lemma \ref{malglem}.  Now assume that the proposition holds over $R_{n-1}$.  Let $R'$ be the formal power series ring $k[[x_1, \ldots, x_{n-2}, x_n]]$, and define a differential operator $\Delta': R' \rightarrow R'$ by
\[
\Delta' = \sum_{i=0}^l r_i(x_1, \ldots, x_{n-2}, 0, x_n) \partial^i.
\]
Since $r_l$ is $x_n$-regular by hypothesis, so is $r_l(x_1, \ldots, x_{n-2}, 0, x_n)$.  Therefore, by the inductive hypothesis, $R'/\Delta'(R')$ is a finitely generated $R_{n-2}$-module.  As we have isomorphisms
\[
\frac{R'}{\Delta'(R')} \simeq \frac{R}{x_{n-1}R + \Delta'(R)} \simeq \frac{R}{\Delta(R) + x_{n-1}R}
\]
of $R'$-modules (and hence of $R_{n-2}$-modules) it follows that $R/(\Delta(R) + x_{n-1}R)$ is a finitely generated $R_{n-2}$-module.  Let $f_1, \ldots, f_m \in R$ be such that $R \subset \sum_{j=1}^m R_{n-2} \cdot f_j + x_{n-1}R + \Delta(R)$.  If $r \in R$ is given, then there exist $\alpha_1, \ldots, \alpha_m \in R_{n-2}$ and $g,h \in R$ such that
\[
r = \sum_{j=1}^m \alpha_j f_j + x_{n-1}g + \Delta(h).
\]
Likewise, since $g \in R$, there exist $\beta_1, \ldots, \beta_m \in R_{n-2}$ and $g',h' \in R$ such that $g = \sum_{j=1}^m \beta_j f_j + x_{n-1}g' + \Delta(h')$.  Substituting this expression for $g$ in the previous equation gives
\[
r = \sum_{j=1}^m (\alpha_j + \beta_j x_{n-1})f_j + x_{n-1}^2 g' + \Delta(h + x_{n-1}h'),
\]
and we can find a similar expression to substitute for $g'$ and continue.  Since $R$ and $R_{n-1}$ are Noetherian complete local rings and $x_{n-1}$ belongs to their maximal ideals, this process converges to an expression
\[
r = \sum_{j=1}^m \rho_j f_j + 0 + \Delta(\eta)
\]
where $\eta \in R$ and all $\rho_j \in R_{n-1}$; that is, we have $R \subset \sum_{j=1}^m R_{n-1} \cdot f_j + \Delta(R)$, so $R/\Delta(R)$ is a finitely generated $R_{n-1}$-module, completing the proof.
\end{proof}

The following is the key technical result in our study of the cokernel of $\partial$ acting on a holonomic, $x_n$-regular $\D$-module:

\begin{lem}[Key lemma for cokernels]\cite[Cor. 2]{cokernel}\label{coklemma}
Let $M$ be a $\D$-module, and let $m \in M$ be an $x_n$-regular element.  There exists $p \geq 1$ and a finitely generated $R_{n-1}$-submodule $E_0$ of $R \cdot m$ such that $R \cdot m \subset E_0 + \sum_{i=1}^p \partial^i(R \cdot m)$.  In particular, $R \cdot m \subset E_0 + \partial(M)$.
\end{lem}

We remark that $E_0$ can be taken to be the $R_{n-1}$-submodule generated by $m, x_n m, \ldots, x_n^N m$ for some $N$ \cite[Thm. 3.2]{several}; however, we will not need this more precise statement, and its proof is more complicated than the proof below.

\begin{proof}
By the definition of $x_n$-regularity of $m$, there exists an $x_n$-regular $f \in R$ such that $E_{f\partial}(m)$ is a finitely generated $R$-module.  Write $E$ for $E_{f\partial}(m)$ and $\tau$ for the derivation $f\partial: R \rightarrow R$.  The finite generation of $E_{\tau}(m)$ over $R$ implies that for some $p$, $\tau^p(m) \in \sum_{i=0}^{p-1} R \cdot \tau^i(m)$ (so that $E$ can be identified with the $R$-module on the right-hand side).  We claim that $E/\tau(E)$ is a finitely generated $R_{n-1}$-module.  Let $r_0, \ldots, r_{p-1} \in R$ be such that $\tau^p(m) = r_0m + r_1\tau(m) + \cdots + r_{p-1}\tau^{p-1}(m)$.  Define $\delta: R \rightarrow R$ to be the differential operator $\tau^p - \sum_{i=0}^{p-1} r_i\tau^i$.  Then $\delta(m) = 0$, and therefore $E$ is a quotient of the $R$-module $N = R\langle \tau \rangle/(R\langle\tau\rangle \cdot \delta)$.  It will suffice to show that $N/\tau(N)$ is a finitely generated $R_{n-1}$-module.  As we have isomorphisms
\[
\frac{N}{\tau(N)} \simeq \frac{R\langle\tau\rangle/(R\langle\tau\rangle \cdot \delta)}{\tau(R\langle\tau\rangle/(R\langle\tau\rangle \cdot \delta))} \simeq \frac{R\langle\tau\rangle}{R\langle\tau\rangle \cdot \delta + \tau R\langle\tau\rangle} \simeq \frac{R}{R \cap (R \cdot \delta + \tau R\langle\tau\rangle)}
\]
of $R_{n-1}$-modules, it suffices to show $R/(R \cap (R \cdot \delta + \tau R\langle\tau\rangle))$ is a finitely generated $R_{n-1}$-module.  We claim that this module can be identified with $R/\Delta(R)$ for some choice of differential operator $\Delta$ satisfying the hypothesis of Proposition \ref{weiergen}.

To this end, we introduce a ``transposition'' operation on the ring $R\langle\tau\rangle$.  Let $\phi: R\langle\tau\rangle \rightarrow R\langle\tau\rangle$ be the unique additive map such that $\phi(\tau) = -\tau$, $\phi(g) = g$ for all $g \in R$, and $\phi(ST) = \phi(T)\phi(S)$ for all $S,T \in R\langle\tau\rangle$.  We will use the notation $S^*$ for $\phi(S)$.  For all $g \in R$ and $S \in R\langle\tau\rangle$, we have $gS \equiv S^*g$ mod $\tau R\langle\tau\rangle$ (this follows by induction on the $\tau$-degree of $S$: if $\deg_{\tau}(S) = 0$, then $S^* = S \in R$ and the statement is immediate).  Therefore
\[
R \cap (R \cdot \delta + \tau R\langle\tau\rangle) = \delta^*R.
\]  
Since $(\tau^p)^* = (-1)^p\tau^p$, the leading term of $\delta^*$ with respect to $\partial$ is $(-f)^p\partial^p$.  By hypothesis, $f$ is $x_n$-regular, and so $(-f)^p$ is $x_n$-regular as well.  It follows that $\Delta = \delta^*$ satisfies the hypotheses of Proposition \ref{weiergen}, and so
\[
\frac{N}{\tau(N)} \simeq \frac{R}{R \cap (R \cdot \delta + \tau R\langle\tau\rangle)} \simeq \frac{R}{\Delta(R)}
\]
is a finitely generated $R_{n-1}$-module.  As explained above, this proves that $E/\tau(E) = E/f \partial(E)$ is finitely generated over $R_{n-1}$.

The ring $R_{n-1}$ is Noetherian, so the $R_{n-1}$-submodule
\[
(fE + f\partial (E))/f\partial (E) \subset E/f \partial(E)
\] 
is also finitely generated.  This implies that there exists a finitely generated $R_{n-1}$-submodule $G \subset E$ such that $fE \subset fG + f\partial(E)$.  Now define $K = (0 :_{E + \partial(E)} f)$.  Clearly $K$ is a finitely generated $R$-module, since $E$ is.  As $K$ is annihilated by $f$, $K$ is in fact a finitely generated $R/fR$-module.  Since $f$ is $x_n$-regular by assumption, it follows from the Weierstrass preparation theorem that $K$ is a finitely generated $R_{n-1}$-module.  

We assert now that $E \subset G + K + \partial(E)$.  Given $\mu \in E$, we have $f\mu \in fE \subset fG + f\partial(E)$, so there exist $\mu' \in G$ and $\mu'' \in E$ such that $f\mu = f\mu' + f\partial(\mu'')$.  Rewrite this equation as 
\[
f(\mu-\mu') + f\partial(-\mu'') = f((\mu-\mu') + \partial(-\mu'')) = 0
\] 
to see that, by definition, $(\mu-\mu') + \partial(-\mu'')$ is an element of $K$ (since $\mu-\mu' \in E$).  Then $\mu = \mu' + ((\mu-\mu') + \partial(-\mu'')) + \partial(\mu'') \in G + K + \partial(E)$, as desired.  If we define $E_0 = G + K$, a finitely generated $R_{n-1}$-module (being the sum of two such), we have $E \subset E_0 + \partial(E)$, so that $R \cdot m \subset E \subset E_0 + \partial(E)$.  

By the Leibniz rule, it is straightforward to see that 
\[
\partial(E) = \partial(\sum_{i=0}^{p-1} R \cdot (f\partial)^i(m)) \subset \sum_{i=1}^p \partial^i(R \cdot m),
\] 
so all that remains to check is that the finitely generated $R_{n-1}$-module $E_0$, which we have shown satisfies $R \cdot m \subset E_0 + \sum_{i=1}^p \partial^i(R \cdot m)$, can be replaced with another finitely generated $R_{n-1}$-module which is actually a submodule of $R \cdot m$.  This can be done since every element of $E_0$ is a sum of terms belonging either to $E = \sum_{i=0}^{p-1} R \cdot (f\partial)^i(m)$ or to $\partial(E)$ and is thus congruent modulo $\sum_{i=1}^p \partial^i(R \cdot m)$ to an element of $R \cdot m$: given a finite set of $R_{n-1}$-generators for $E_0$, if we replace each of them by an element of $R \cdot m$ in the same $(\sum_{i=1}^p \partial^i(R \cdot m))$-congruence class and then replace $E_0$ with the $R_{n-1}$-module having these new generators, we obtain the statement of the proposition. 
\end{proof}

\section{Cokernels}\label{cokernels}

In light of Proposition \ref{holkernel}, we may be led to conjecture that the analogue holds for cokernels: that if $M$ is a holonomic $\D$-module, then $\overline{M} = \coker(\partial: M \rightarrow M)$ is a holonomic $\D_{n-1}$-module.  This conjecture is false:

\begin{prop}\cite[Thm.]{example}\label{counterex}
There exists a holonomic $\D$-module $M$ such that $\overline{M} = M/\partial(M)$ is not a holonomic $\D_{n-1}$-module.  Specifically, take $n=4$, let $f = x_1x_4 + x_2 + x_3x_4e^{x_4}$, and define $M = R_f$.  Then $M/\partial_4(M)$ is not a holonomic $\D_3$-module.
\end{prop}

As the details of this counterexample are not necessary for the proof of Theorem \ref{mainthm}, we omit them here, contenting ourselves with the following outline.

\begin{enumerate}
\item If $f = x_1x_4 + x_2 + x_3x_4e^{x_4}$, then $R_3 \cap (R \cdot f + R \cdot \partial_4 f) = 0$ (a tricky but explicit calculation).
\item For any $f \in R$, if $R_f/\partial_4(R_f)$ is a holonomic $\D_3$-module, then there exists a nonzero $g \in R_3$ such that $(R_f/\partial_4(R_f))_g$ is a finitely generated $(R_3)_g$-module (by Proposition \ref{fgloc}).
\item For any $f \in R$, if $f$ is irreducible and coprime to $\partial_4(f)$ and $(R_f/\partial_4(R_f))_g$ is a finitely generated $(R_3)_g$-module for some nonzero $g \in R_3$, then
\[
R_3 \cap (R \cdot f + R \cdot \partial_4 f) \neq 0
\] 
(another tricky calculation).
\item If $f = x_1x_4 + x_2 + x_3x_4e^{x_4}$, then $f$ is irreducible and coprime to $\partial_4(f)$.  It follows that if $R_f/\partial_4(R_f)$ were a holonomic $\D_3$-module, (1), (2), and (3) above would produce a contradiction.
\end{enumerate}

It follows that the analogue of Proposition \ref{holkernel} will only hold under additional hypotheses on $M$.  In this section, we will show that an $x_n$-regularity hypothesis suffices.  If $M = \D \cdot m$ is a holonomic $\D$-module, there is a natural choice of good filtration on $M$ defined using the chosen generator $m$ and the order filtration on $\D$.  This filtration descends to a filtration on the cokernel $\overline{M} = M/\partial(M)$.  The import of our key lemma for cokernels, Lemma \ref{coklemma}, is that if $m$ is an \emph{$x_n$-regular} element, this filtration is also good:

\begin{lem}\cite[Cor. 3]{cokernel}\label{goodfil}
Let $M = \D \cdot m$ be a holonomic $\D$-module, and suppose that $m$ is an $x_n$-regular element.  Let $\{M_j\}$ be the good filtration on $M$ defined by $M_j = \D_j \cdot m$ for all $j$, and let $\{\overline{M}_j\}$ be the corresponding filtration on $\overline{M} = M/\partial(M)$ defined by $\overline{M}_j = (M_j + \partial(M))/\partial(M)$ for all $j$.  Then $\overline{M}$ is a finitely generated $\D_{n-1}$-module, and $\{\overline{M}_j\}$ is a good filtration on $\overline{M}$.
\end{lem}

\begin{proof}
By Lemma \ref{coklemma}, there exists $p \geq 1$ and a finitely generated $R_{n-1}$-submodule $E_0$ of $R \cdot m$ such that $R \cdot m \subset E_0 + \partial(M)$.  Suppose that $m_1, \ldots, m_l$ are $R_{n-1}$-generators of $E_0$, so that $R \cdot m \subset (\sum_{i=1}^l R_{n-1} \cdot m_i) + \partial(M)$.  Then if $\overline{m_i}$ is the class of $m_i$ in $\overline{M}$, we have $\overline{M} = \sum_{i=1}^l \D_{n-1} \cdot \overline{m_i}$ (so $\overline{M}$ is a finitely generated $\D_{n-1}$-module) and $\overline{M}_j = \sum_{i=1}^l (\D_{n-1})_j \overline{m_i}$, so that $\{\overline{M}_j\}$ is a good filtration (the associated graded $\gr \D_{n-1}$-module is generated by the images of the $\overline{m_i}$).
\end{proof}

We can now state and prove the analogue of Proposition \ref{holkernel} with a suitable additional hypothesis:

\begin{prop}\cite[Cor. 1.7]{essen}\label{coker}
Let $M$ be a holonomic, $x_n$-regular $\D$-module.  Then $\overline{M} = M/\partial(M)$ is a holonomic $\D_{n-1}$-module.
\end{prop}

\begin{proof}
Fix $m \in M$ such that $M = \D \cdot m$.  Let $L \subset \D$ be the annihilator of $m$ in $\D$ (a left ideal) so that $M \simeq \D/L$ as (left) $\D$-modules.  By hypothesis, there exists an $x_n$-regular $f \in R$ such that $E_{f\partial}(m)$ is a finitely generated $R$-module.  Let $\{M_j\}$ and $\{\overline{M}_j\}$ be the good filtrations defined (on $M$ and $\overline{M}$, respectively) in the statement of Lemma \ref{goodfil}.

We now consider various associated graded objects.  Let $S_{n-1} = R_{n-1}[\zeta_1, \ldots, \zeta_{n-1}]$ be the associated graded ring $\gr \D_{n-1}$ of $\D_{n-1}$ with respect to the order filtration, where $\zeta_i$ is the principal symbol of $\partial_i$.  Similarly, let $S = \gr \D = R[\zeta_1, \ldots, \zeta_n]$.  Let $\sigma(L) \subset S$ be the ideal generated by the principal symbols of the elements of $L$.  Let $\gr M$ be the associated graded $S$-module of $M$ with respect to the filtration $\{M_j\}$, and $\gr \overline{M}$ the associated graded $S_{n-1}$-module of $\overline{M}$ with respect to the filtration $\{\overline{M}_j\}$.  Since $M \simeq \D/L$, we have $\gr M \simeq \gr(\D/L) \simeq S/\sigma(L)$ as $S$-modules.  Consider the natural surjective map $\gr M \rightarrow \gr \overline{M}$ defined by associating, to the class of an element of $M_j$ modulo $M_{j-1}$, its class modulo $M_{j-1} + \partial(M)$.  This map is $S_{n-1}$-linear, and as the principal symbol of $\partial$ is $\zeta_n$, it is clear that $\zeta_n \gr M$ lies in its kernel.  We therefore obtain an $S_{n-1}$-linear surjection 
\[
\gr M/\zeta_n \gr M \rightarrow \gr \overline{M}.
\]  
Since $\gr M \simeq S/\sigma(L)$, the source of this surjection can be identified with $S/(\sigma(L) + (\zeta_n))$ as an $S$-module, and since this surjection is $S_{n-1}$-linear, we see that
\[
S_{n-1} \cap (\sigma(L) + (\zeta_n)) \subset \Ann_{S_{n-1}} \gr \overline{M},
\]
where both sides are ideals of $S_{n-1}$.  Therefore we have
\[
d(\overline{M}) = \dim \, S_{n-1}/(\Ann_{S_{n-1}} \gr \overline{M}) \leq \dim \, S_{n-1}/(S_{n-1} \cap (\sigma(L) + (\zeta_n))),
\]
where the equality holds by the definition of dimension of a $\D_{n-1}$-module (as $\{\overline{M}_j\}$ is a good filtration) and the inequality follows from the containment of ideals above.  

We now state two basic facts about radicals of ideals, whose proofs are immediate and which hold in general for commutative rings.  Let $\mathfrak{a}$ and $\mathfrak{b}$ be ideals of $S$.  Then
\[
\sqrt{\mathfrak{a} + \mathfrak{b}} = \sqrt{\sqrt{\mathfrak{a}} + \mathfrak{b}} = \sqrt{\sqrt{\mathfrak{a}} + \sqrt{\mathfrak{b}}},
\]
and if $\sqrt{\mathfrak{a}} = \sqrt{\mathfrak{b}}$, then $\sqrt{S_{n-1} \cap \mathfrak{a}} = \sqrt{S_{n-1} \cap \mathfrak{b}}$.  Together, these facts imply that if $J = \sqrt{\sigma(L)}$, then
\begin{align*}
\dim \, S_{n-1}/(S_{n-1} \cap (\sigma(L) + (\zeta_n))) &= \dim \, S_{n-1}/(S_{n-1} \cap (J + (\zeta_n)))\\ &= \dim \, S_{n-1}/(S_{n-1} \cap \sqrt{J + (\zeta_n)}),
\end{align*}
since the three denominators all have the same radical.  In particular, we have
\[
d(\overline{M}) \leq \dim \, S_{n-1}/(S_{n-1} \cap (J + (\zeta_n))).
\]
Note that $J = \sqrt{\Ann_S \gr M}$.  Since $M$ is a holonomic $\D$-module, $d(M) = \dim \, S/J = n$.  By Bernstein's inequality, $d(\overline{M}) \geq n-1$, and so it suffices to prove that $d(\overline{M}) \leq n-1 = \dim \, S/J - 1$.  That is, we have reduced ourselves to showing that
\[
\dim \, S_{n-1}/(S_{n-1} \cap (J + (\zeta_n))) \leq \dim \, S/J - 1.
\]
As $J \subset S$ is a radical ideal in a Noetherian ring, we can write it as an intersection of finitely many prime ideals: let $J = \mathfrak{p}_1 \cap \cdots \cap \mathfrak{p}_t$ be such an expression where each $\mathfrak{p}_i$ is minimal over $J$.  We claim that
\[
\sqrt{J + (\zeta_n)} = \cap_{i=1}^t \sqrt{\mathfrak{p}_i + (\zeta_n)}.
\]
One containment is obvious.  For the other, let $x$ belong to $\sqrt{\mathfrak{p}_i + (\zeta_n)}$ for all $i$, and suppose $m$ is large enough that $x^m \in \mathfrak{p}_i + (\zeta_n)$ for all $i$.  Then there exist $y_i \in \mathfrak{p}_i$ and $z_i \in R$ such that
\[
x^m = y_1 + z_1\zeta_n = \cdots = y_t + z_t\zeta_n,
\]
and so $x^{mt} - y_1 \cdots y_t \in (\zeta_n)$.  Since $y_1 \cdots y_t \in \cap_{i=1}^t \mathfrak{p}_i = J$, it follows that $x \in \sqrt{J + (\zeta_n)}$, as claimed.  Therefore,
\begin{align*}
\dim \, S_{n-1}/(S_{n-1} \cap (J + (\zeta_n))) &= \dim \, S_{n-1}/(S_{n-1} \cap \sqrt{J + (\zeta_n)})\\ &= \dim \, S_{n-1}/(S_{n-1} \cap (\cap_{i=1}^t \sqrt{\mathfrak{p}_i + (\zeta_n)})),
\end{align*}
and so it suffices to prove
\[
\dim \, S_{n-1}/(S_{n-1} \cap (\mathfrak{p} + (\zeta_n))) \leq \dim \, S/\mathfrak{p} - 1
\]
where $\mathfrak{p}$ is a minimal prime ideal containing $J$.

First suppose that $f \in \mathfrak{p}$.  As $f$ is $x_n$-regular, $R/fR$ is a finitely generated $R_{n-1}$-module by the Weierstrass preparation theorem.  Therefore, since $f \in \mathfrak{p}$, $S/\mathfrak{p}$, and, \emph{a fortiori}, $S/(\mathfrak{p} + (\zeta_n))$, is a finitely generated $S_{n-1}$-module.  It follows that
\[
S_{n-1}/(S_{n-1} \cap (\mathfrak{p} + (\zeta_n))) \subset S/(\mathfrak{p} + (\zeta_n))
\]
is a finite (hence integral) extension of Noetherian rings, and so both rings have the same dimension \cite[Ex. 9.2]{matsumura}.  We have therefore reduced ourselves to proving that
\[
\dim \, S/(\mathfrak{p} + (\zeta_n)) \leq \dim \, S/\mathfrak{p} - 1,
\]
that is (since $S/\mathfrak{p}$ is an integral domain), that $\zeta_n \notin \mathfrak{p}$.  Suppose for contradiction that $\zeta_n \in \mathfrak{p}$.  The ideal $J = \sqrt{\Ann_S \gr M}$ is involutive by Gabber's theorem, and so $\mathfrak{p}$ is also involutive by Corollary \ref{primeinvol}.  Since $f$ and $\zeta_n$ both belong to $\mathfrak{p}$, so does the Poisson bracket $\{\zeta_n, f\} = \partial(f)$; continuing in this way, $\partial^l(f) \in \mathfrak{p}$ for all $l$.  Taking $l$ to be the smallest index such that $x_n^l$ appears in the expansion of $f$ with a nonzero scalar coefficient (such $l$ exists since $f$ is $x_n$-regular), we see that $\mathfrak{p}$ contains a unit, a contradiction.  Therefore $\zeta_n \notin \mathfrak{p}$, as desired.

For the other (harder) case, suppose that $f \notin \mathfrak{p}$.  Recall that by hypothesis $E_{f \partial}(m)$ is a finitely generated $R$-module, so there exists $q$ such that $(f \partial)^q(m)$ belongs to the $R$-submodule of $M$ generated by $\{(f \partial)^j(m)\}_{j=0}^{q-1}$.  Let $\rho_0, \ldots, \rho_{q-1} \in R$ be such that
\[
(f \partial)^q(m) = \sum_{j=0}^{q-1} \rho_j (f \partial)^j(m);
\]
it follows that $(f \partial)^q - \sum_{j=0}^{q-1} \rho_j (f \partial)^j \in \D$ annihilates $m$, and hence its principal symbol $(f \zeta_n)^q$ belongs to $\sigma(L)$.  Therefore $f\zeta_n \in \sqrt{\sigma(L)} = J \subset \mathfrak{p}$.  As we have assumed that $f \notin \mathfrak{p}$ and $\mathfrak{p}$ is prime, this implies that $\zeta_n \in \mathfrak{p}$.

For any $\alpha \in \mathfrak{p}$, let $\alpha_{\circ} = \alpha(x_1, \ldots, x_{n-1}, 0, \zeta_1, \ldots, \zeta_{n-1}, 0) \in S_{n-1}$, and let $\mathfrak{p}_{\circ}$ be the ideal of $S_{n-1}$ consisting of all $\alpha_{\circ}$ where $\alpha$ ranges over $\mathfrak{p}$.  We have
\[
\mathfrak{p}_{\circ} + (x_n, \zeta_n) = \mathfrak{p} + (x_n, \zeta_n) = \mathfrak{p} + (x_n)
\]
as ideals of $S$, since $\zeta_n \in \mathfrak{p}$.  We note that $\mathfrak{p} + (x_n) \neq S$ as a consequence of the fact that $\mathfrak{p}$ is homogeneous with respect to $\zeta_1, \ldots, \zeta_n$ (Remark \ref{zetahom}): if $1 + sx_n \in \mathfrak{p}$ for some $s \in S$, then $1 + s_0x_n \in \mathfrak{p}$ where $s_0 \in R$ is the constant term of $s$ with respect to the $\zeta_i$; but then $\mathfrak{p}$ contains a unit of $R$, a contradiction.

It is clear that $\mathfrak{p} \cap S_{n-1} \subset \mathfrak{p}_{\circ}$, and we claim that equality holds, that is, that $\mathfrak{p}_{\circ} \subset \mathfrak{p}$.  Since $\zeta_n \in \mathfrak{p}$, it suffices to check that if $a \in \mathfrak{p}$ is of the form $\sum_{i=0}^{\infty} a_i x_n^i$ with $a_i \in S_{n-1}$, then the $x_n$-constant term $a_0$ belongs to $\mathfrak{p}$.  We will verify this by showing that $a_0 \in \mathfrak{p} + x_n^q S'$ for all $q \geq 1$, where $S' = R[\zeta_1, \ldots, \zeta_{n-1}] \subset S$.  This suffices because then 
\[
a \in \cap_{q=1}^{\infty} \mathfrak{p} + (x_n^q) \subset S,
\]
and the right-hand side is simply $\mathfrak{p}$ by Krull's intersection theorem \cite[Thm. 8.10(ii)]{matsumura} applied to the integral domain $S/\mathfrak{p}$ and its ideal $(\mathfrak{p} + (x_n))/\mathfrak{p}$, which is a proper ideal since we have already checked that $\mathfrak{p} + (x_n) \neq S$.

It is clear that $a_0 \in \mathfrak{p} + x_nS'$.  Now assume for some $q \geq 1$ that $a_0 = g + x_n^q h$ for some $g \in \mathfrak{p}$ and $h \in S'$, and let $h_0$ be the $x_n$-constant term of $h$.  On the one hand, since $a_0$ belongs to $S'$, the Poisson bracket $\{\zeta_n, a_0\}$ is zero.  On the other hand, using the biderivation property, we see that
\[
0 = \{\zeta_n, a_0\} = \{\zeta_n, g\} + x_n^q\{\zeta_n, h\} + qx_n^{q-1}h.
\]
Since $\mathfrak{p}$ is involutive by Corollary \ref{primeinvol}, we have $\{\zeta_n, g\} \in \mathfrak{p}$, from which it follows that $qx_n^{q-1}h_0 \in \mathfrak{p} + x_n^qS'$, hence $x_n^qh_0 \in \mathfrak{p} + x_n^{q+1}S'$, and finally $a_0 = g + x_n^qh \in \mathfrak{p} + x_n^{q+1}S'$, completing the induction.  We conclude that $\mathfrak{p}_{\circ} = \mathfrak{p} \cap S_{n-1}$.

We can now finish the proof.  We have isomorphisms of rings
\[
\frac{S_{n-1}}{S_{n-1} \cap (\mathfrak{p} + (\zeta_n))} \simeq \frac{S_{n-1}}{\mathfrak{p}_{\circ}} \simeq \frac{S}{\mathfrak{p}_{\circ} + (x_n, \zeta_n)} \simeq \frac{S}{\mathfrak{p} + (x_n)},
\]
and hence $\dim \, S_{n-1}/(S_{n-1} \cap (\mathfrak{p} + (\zeta_n))) = \dim \, S/(\mathfrak{p} + (x_n))$.  We need only show that $\dim \, S/(\mathfrak{p} + (x_n)) \leq \dim \, S/\mathfrak{p} - 1$, that is (since $S/\mathfrak{p}$ is an integral domain) that $x_n \notin \mathfrak{p}$.  But this is immediate: if $x_n \in \mathfrak{p}$, then since $\mathfrak{p}$ is involutive, $\{\zeta_n, x_n\} = 1 \in \mathfrak{p}$, a contradiction.  This completes the proof.
\end{proof}

We now have all we need for the proof of Theorem \ref{mainthm}:

\begin{proof}[Proof of Theorem \ref{mainthm}]
We proceed by induction on $n$.  The case $n=0$ is obvious, since a holonomic $\D_0$-module is nothing but a finite-dimensional $k$-space.  Now suppose that $n > 0$ and $M$ is a holonomic $\D$-module.  By Proposition \ref{fgloc}, there exists a nonzero $f \in R$ such that $M_f$ is a finitely generated $R_f$-module.  After a linear change of coordinates, we may assume $f$ is $x_n$-regular; by Proposition \ref{dRind}, this change of coordinates does not affect the de Rham cohomology of $M$.  Assuming this change of coordinates has been made, $M$ is $x_n$-regular by Lemma \ref{reglink}.  By Propositions \ref{holkernel} and \ref{coker}, the kernel $M_*$ and cokernel $\overline{M}$ of $\partial$ acting on $M$ are holonomic $\D_{n-1}$-modules, and by the inductive hypothesis have finite-dimensional de Rham cohomology.  The exact sequences
\[
\cdots \rightarrow H_{dR}^i(M_*) \rightarrow H_{dR}^i(M) \rightarrow H_{dR}^{i-1}(\overline{M}) \rightarrow \cdots
\]
of Lemma \ref{derhamles} finish the proof.
\end{proof}

\end{document}